\documentclass[a4paper,leqno,10pt]{amsart}

\usepackage{epsf,graphicx,epsfig}
\usepackage{amscd}
\usepackage{amsmath,latexsym,amssymb,amsthm}
\usepackage[nospace,noadjust]{cite}
\usepackage{textcomp}
\usepackage{setspace,cite}
\usepackage{lscape,fancyhdr,fancybox}
\usepackage{stmaryrd}
\usepackage[all,cmtip]{xy}
\usepackage{tikz}
\usepackage{cancel}
\usetikzlibrary{shapes,arrows,decorations.markings}
\usepackage{graphicx}

\usepackage{color}
\usepackage{url}
\usepackage{enumerate}
\usepackage[mathscr]{euscript}
  \theoremstyle{plain}

\swapnumbers
    \newtheorem{thm}{Theorem}[section]
    \newtheorem{prop}[thm]{Proposition}

    \newtheorem{subsec}[thm]{}
\theoremstyle{definition}
    \newtheorem{defn}[thm]{Definition}
        \newtheorem{remark}[thm]{Remark}
\theoremstyle{remark}

\title{}
\author{}
\date{}
\usepackage{amssymb}

\usepackage{hyperref}
\hypersetup{
	colorlinks,
	citecolor=blue,
	filecolor=black,
	linkcolor=blue,
	urlcolor=black
}

\begin{document}
\title{Crossed extensions of Lie algebras}
\author{Apurba Das}
\address{Department of Mathematics and Statistics,
Indian Institute of Technology, Kanpur 208016, Uttar Pradesh, India}
\email{apurbadas348@gmail.com}


%\maketitle
\begin{abstract}
It is known that Hochschild cohomology groups are represented by crossed extensions of associative algebras. In this paper, we introduce crossed $n$-fold extensions of a Lie algebra $\mathfrak{g}$ by a module $M$, for $n \geq 2$. The equivalence classes of such extensions are represented by the $(n+1)$-th Chevalley-Eilenberg cohomology group $H^{n+1}_{CE} (\mathfrak{g}, M).$
\end{abstract}

\noindent

\thispagestyle{empty}

\maketitle

%\tableofcontents

\vspace{0.2cm}

\section{Introduction}
Group cohomology of a group $G$ with coefficients in a $G$-module $M$ is related to crosses extensions of $G$ by $M$. More precisely, equivalence classes of crossed $n$-fold extensions of $G$ by $M$ are classified by the $(n+1)$-th group cohomology $H^{n+1}(G,M)$ \cite{hueb}. A similar result for Hochschild cohomology was considered by Baues and Minian \cite{baues-minian}. Namely, they introduce crossed $n$-fold extensions of an associative algebra $A$ by an $A$-bimodule $M$ and prove that equivalence classes of such extensions are isomorphic to the $(n+1)$-th Hochschild cohomology $HH^{n+1}(A,M)$ as abelian groups. A further generalization of this result has been obtained in \cite{baues-minian-Richter}.

The cohomology of a Lie algebra $\mathfrak{g}$ with coefficients in a $\mathfrak{g}$-module $M$ is given by the Chevalley-Eilenberg cohomology. This cohomology theory controls the deformation of a given Lie algebra structure. The Chevalley-Eilenberg cohomology groups $H^\bullet_{CE} (\mathfrak{g}, M)$ are also related to special type of $L_\infty$-algebras \cite[Theorem 6.7]{baez-crans}. In this paper, we give another interpretation of Chevalley-Eilenberg cohomology in terms of crossed extensions. The idea is exactly same as Baues and Minian for associative algebra case. We introduce crossed $n$-fold extensions of a Lie algebra $\mathfrak{g}$ by a $\mathfrak{g}$-module $M$. The equivalence classes of such extensions are classified by $H^{n+1}_{CE} (\mathfrak{g}, M).$ We also find a similar result for Leibniz algebras.

 All vector spaces are over a field $\mathbb{K}$.

\section{Chevalley-Eilenberg cohomology}

Let $(\mathfrak{g},[~,~])$ a Lie algebra. A module over $\mathfrak{g}$ consists of a vector space $M$ together with a $\mathbb{K}$-bilinear map $[~,~]: \mathfrak{g} \times M \rightarrow M$ satisfying
\begin{align*}
[[x,y],m] = [x,[y,m]] - [y,[x,m]],
\end{align*}
for $x, y \in \mathfrak{g}$ and $m \in M$. It is clear that $\mathfrak{g}$ is a  module over $\mathfrak{g}$ with respect to the Lie bracket.

Given a Lie algebra $(\mathfrak{g},[~,~])$ and a module $M$, the corresponding Chevalley-Eilenberg cochain groups $\{ C^n_{CE} (\mathfrak{g}, M) \}_{n \geq 0}$ are given by $ C^0_{CE} (\mathfrak{g}, M) = M$ and $ C^n_{CE} (\mathfrak{g}, M) = Hom_{\mathbb{K}} (\wedge^n \mathfrak{g}, M)$, for $n \geq 1$. The coboundary map $\delta :  C^n_{CE} (\mathfrak{g}, M) \rightarrow  C^{n+1}_{CE} (\mathfrak{g}, M)$ is given by
\begin{align*}
(\delta m)(x) =~& [x, m],~~~ \text{ for } m \in M \text{ and } x \in \mathfrak{g},\\
(\delta f) (x_1, \ldots, x_{n+1}) =~& \sum_{i=1}^{n+1} (-1)^{i+1} [x_i, f (x_1, \ldots, \widehat{x_i}, \ldots, x_{n+1}) \\
~&+ \sum_{1 \leq i <  j \leq n+1} (-1)^{i+j} f ([x_i, x_j], x_1, \ldots, \widehat{x_i}, \ldots, \widehat{x_j}, \ldots, x_{n+1}),
\end{align*}
for $f \in C^n_{CE} (\mathfrak{g}, M)$ and $x_1, \ldots, x_{n+1} \in \mathfrak{g}$. The corresponding cohomology groups are denoted by $H^\bullet_{CE} (\mathfrak{g}, M)$ and called the Chevalley-Eilenberg cohomology of $\mathfrak{g}$ with coefficients in the module $M$.

It is easy to see that $H^0_{CE}(\mathfrak{g}, M)$ is the submodule of invariants of $M$:
\begin{align*}
H^0_{CE} (\mathfrak{g}, M) = \{ m \in M |~ [x,m] = 0, \forall x \in \mathfrak{g} \}.
\end{align*}
The first cohomology group $H^1_{CE} (\mathfrak{g}, M)$ can be seen as the space of outer derivations
\begin{align*}
H^1_{CE} (\mathfrak{g}, M) = \text{OutDer} (\mathfrak{g}, M) := \frac{\text{Der} (\mathfrak{g}, M)}{\text{InnDer} (\mathfrak{g}, M)},
\end{align*}
where a derivation is a map $f : \mathfrak{g} \rightarrow M$ satisfying $f [x,y] = [x, fy] - [y, fx]$, for all $x, y \in \mathfrak{g}$ and is called inner if $f(x) = [x, m]$, for some $m \in M$.

The second cohomology group $H^2_{CE} (\mathfrak{g}, M)$ can be described as the space of equivalence classes of Lie algebra extensions
%\begin{align*}
%0 \rightarrow M \rightarrow \mathfrak{h} \rightarrow \mathfrak{g} \rightarrow 0
%\end{align*}
of $\mathfrak{g}$ by the module $M$.

\medskip

\noindent \textbf{Pushout of $\mathfrak{g}$-modules.}
Pushout in a category is an important tool to study some nice properties of the category.
\begin{defn}
Let $\mathcal{C}$ be a category. Given two morphisms $f : A \rightarrow B$ and $g : A \rightarrow C$, a pushout (or fibered sum) is a triple $(D, i, j)$ where $D \in \text{Ob}(\mathcal{C})$ and $i : B \rightarrow D$ and $j : C \rightarrow D$ are morphisms in $\mathcal{C}$ such that $jg = if$ that satisfies the following universal property: for any triple $(D', i', j')$ with $j' g = i' f$, there is a unique morphism $\theta : D \rightarrow D'$ making the following diagram commute
\[
\xymatrix{
A \ar[r]^{g} \ar[d]_{f} & C & & A \ar[r]^{g} \ar[d]_{f} & C \ar[d]^{j} \ar@/^/[ddr]^{j'} &\\
B & & & B \ar[r]_{i} \ar@/_/[drr]_{i'} & D \ar@{.>}[rd]^{\theta} & \\
& & & & & D'.
}
\]

\end{defn}

Let $\mathfrak{g}$ be a Lie algebra. Consider the category $ \mathfrak{g}\textbf{-mod}$ of $\mathfrak{g}$-modules and $\mathfrak{g}$-module morphisms. Then the pushout of two maps $f : A \rightarrow B$ and $g : A \rightarrow C$ in $ \mathfrak{g}\textbf{-mod}$ exists.

Let $S = \{ (f(a), - g(a) ) |~ a \in A \}$. Then it is easy to see that $S$ is a $\mathfrak{g}$-submodule of $B \oplus C$. Take $D = \frac{B \oplus C}{S}$ and define $i : B \rightarrow D$ by $i (b) = (b,0)+ S$ and $j : C \rightarrow D$ by $j(c) = (0, c) + S$. It is easy to see that $jg = if$. Moreover, if $(D', i', j')$ is another triple with  $j'g = i' f$, we define $\theta : D \rightarrow D'$ by $\theta ((b,c)+ S) = i'(b) + j'(c)$. It is also easy to check that $\theta$ is unique. Hence the claim.

\section{Crossed modules of Lie algebras}
In this section we give an interpretation of $H^3_{CE} (\mathfrak{g}, M)$ in terms of crossed module over Lie algebras. However our notion of crossed module is different from the traditional one (see for example \cite{kassel-loday, wagemann}). Our definition is motivated from the one given by Baues and Minian for associative algebras \cite{baues-minian}.

\begin{defn}
A crossed module over a Lie algebra is a triple $(V, L, \partial)$ in which $L$ is a Lie algebra, $V$ is a $L$-module and $\partial : V \rightarrow L$ is a map of $L$-modules such that
\begin{align*}
[\partial v, w] = - [ \partial w, v ], ~~ \text{ for } v, w \in V.
\end{align*}
\end{defn}

\begin{defn}
Let $(V, L, \partial)$ and $(V', L', \partial')$ be two crossed modules. A map between them consists of a linear map $\alpha : V \rightarrow V'$ and a Lie algebra map $\beta : L \rightarrow L'$ such that the following diagram commute
\[
\xymatrix{
V \ar[r]^{\partial} \ar[d]_{\alpha} & L \ar[d]^{\beta} \\
V' \ar[r]_{\partial'} & L'
}
\]
and satisfying $\alpha [x,v] = [\beta(x), \alpha (v)],$ for $x \in L,~ v \in V.$
\end{defn}

Given a crossed module $(V , L, \partial)$, we consider $\mathfrak{g} = \text{coker}(\partial)$ and $M = \text{ker} (\partial)$. The Lie algebra structure on $L$ induces a Lie algebra structure on $\mathfrak{g}$ by $[\pi(x), \pi(y)] = \pi [x,y]$, where $\pi : L \rightarrow \mathfrak{g}$ is the projection map. Moreover, the action of $L$ on $V$ induces an action of $\mathfrak{g}$ on $M$ via $[\pi (x), m] = [x,m]$, for $m \in M$. Hence a crossed module yields an exact sequence
\begin{align*}
0 \rightarrow M \xrightarrow{i} V \xrightarrow{\partial} L \xrightarrow{\pi} \mathfrak{g} \rightarrow 0.
\end{align*}
We call $(V , L, \partial )$ a crossed module over the Lie algebra $\mathfrak{g}$ with kernel a $\mathfrak{g}$-module $M$. Two crossed modules 
$(V, L, \partial)$ and $(V', L', \partial')$ over $\mathfrak{g}$ with kernel $M$ are said to be equivalent if there is a morphism of crossed modules $(V, L, \partial) \rightarrow (V', L', \partial')$ which induces identity maps on $\mathfrak{g}$ and $M$. Let $\text{Cross}(\mathfrak{g}, M)$ be the equivalence classes of such crossed modules.
In the next theorem we see that equivalence classes of crossed modules are in one-to-one correspondence with the third Chevalley-Eilenberg cohomology.

\begin{thm}\label{cross-3}
There is a bijection
\begin{align*}
\psi : \text{Cross}(\mathfrak{g}, M) \rightarrow H^3_{CE} (\mathfrak{g}, M).
\end{align*}
\end{thm}

\begin{proof}
Let
\begin{align*}
\mathcal{E} : 0 \rightarrow M \xrightarrow{i} V \xrightarrow{\partial} L \xrightarrow{\pi} \mathfrak{g} \rightarrow 0
\end{align*}
be a crossed module. Choose $\mathbb{K}$-linear sections $s : \mathfrak{g} \rightarrow L$ with $\pi s = \text{id}$ and $q : \text{Im}(\partial) \rightarrow V$ with $\partial q = \text{id}$. For any $x , y \in \mathfrak{g},$ we have
\begin{align*}
\pi ([s(x), s(y)] - s[x,y]) = 0.
\end{align*}
This shows that $[s(x), s(y)] - s[x,y]$ is in ker $(\pi) = \text{Im} (\partial)$. Take $g (x,y) = q ([s(x), s(y)] - s[x,y]) \in V$. Define a map $\theta_{\mathcal{E}, s, q} : \mathfrak{g}^{\otimes 3} \rightarrow M$ by
%\begin{align*}
$$\theta_{\mathcal{E}, s, q} (x,y,z) = [s(x), g(y,z)] - [s(y), g(x,z)] + [s(z), g (x,y)] 
 - g ([x,y], z) + g ([x,z], y) - g ([y,z],x).$$
%\end{align*}
Since $\partial$ is a map of $L$-modules, it follows that $\partial (\theta_{\mathcal{E}, s, q} (x,y,z)) = 0.$ Therefore, $\theta_{\mathcal{E}, s, q} (x,y,z) \in \text{ker} (\partial) = M$. The map $\theta_{\mathcal{E}, s, q} : \mathfrak{g}^{\otimes 3} \rightarrow M$ is skew-symmetric in $x, y, z$. Hence $\theta_{\mathcal{E}, s, q} : \wedge^3 \mathfrak{g} \rightarrow M$. The map $\theta_{\mathcal{E}, s, q}$ defines a $3$-cocycle in the Chevalley-Eilenberg cohomology of $\mathfrak{g}$ with coefficients in $M$.

We first show that the class of $\theta_{\mathcal{E}, s, q}$ in $H^3_{CE}(\mathfrak{g}, M)$ does not depend on the section $s$. Suppose $\overline{s} : \mathfrak{g} \rightarrow L$ is another section of $\pi$ and let $\theta_{\mathcal{E}, \overline{s}, q}$ be the corresponding $3$-cocycle defined by using $\overline{s}$ instead of $s$. Then there exists a linear map $h : \mathfrak{g} \rightarrow V$ with $s - \overline{s} = \partial h$. Observe that
\begin{align*}
[s(x), g(y,z)] - [\overline{s}(x) , \overline{g} (y,z)] 
=~& [s(x) - \overline{s}(x), g (y,z)] + [ \overline{s}(x), (g - \overline{g})(y,z)] \\
=~& [\partial h (x) , q ([s(y), s(z)] - s[y,z]) ] + [ \overline{s}(x), (g - \overline{g})(y,z)] \\
=~& - [ ([s(y), s(z)] - s[y,z]), h(x) ] + [ \overline{s}(x), (g - \overline{g})(y,z)].
\end{align*}
Therefore,
\begin{align}\label{e-e'}
( \theta_{\mathcal{E}, s, q} - \theta_{\mathcal{E}, \overline{s}, q}) (x,y,z) = ~&
- [ ([s(y), s(z)] - s[y,z]), h(x) ] + [ ([s(x), s(z)] - s[x,z]), h(y) ] \\
~&- [ ([s(x), s(y)] - s[x,y]), h(z) ]  + [ \overline{s}(x), (g - \overline{g})(y,z)]  \nonumber \\
~& - [ \overline{s}(y), (g - \overline{g})(x,z)] + [ \overline{s}(z), (g - \overline{g})(x,y)] \nonumber \\
~~&- (g - \overline{g}) ([x,y],z) + (g - \overline{g}) ([x,z], y) - (g - \overline{g}) ([y,z], x). \nonumber
\end{align}

Define a map $b : \wedge^2 \mathfrak{g} \rightarrow V$ by
\begin{align*}
b (x,y) = [s(x), h(y)] - h ([x,y]) - [s(y), h(x)] - [ \partial h(x), h (y)].
\end{align*}
Then a easy calculation shows that  $\partial b = \partial (g - \overline{g})$. Hence $(g - \overline{g} - b) : \wedge^2 \mathfrak{g} \rightarrow M$. It follows from (\ref{e-e'}) that
\begin{align*}
( \theta_{\mathcal{E}, s, q} - \theta_{\mathcal{E}, \overline{s}, q}) (x,y,z) =~& - [ ([s(y), s(z)] - s[y,z]), h(x) ] + [ ([s(x), s(z)] - s[x,z]), h(y) ] \\
~& - [ ([s(x), s(y)] - s[x,y]), h(z) ] + (\delta (g - \overline{g} - b)) (x,y,z) \\
~&+ [ \overline{s}(x), b (y, z)] - [ \overline{s}(y), b (x,z)] + [\overline{s}(z) , b (x,y)] \\
~&- b ([x,y],z) + b  ([x,z], y) - b ([y,z], x).
\end{align*}
In the right hand side of the above equation, we substitute the definition of $b$ in last six terms. After many cancellations on the right hand side (one has to use the fact that $(V , L, \partial)$ is a crossed module for some cancellations), we are only left with the term $(\delta (g - \overline{g} - b)) (x,y,z)$. Hence the class of $\theta_{\mathcal{E}, s, q}$ does not depend on the section $s$.

\medskip

Next consider a map
\[
\xymatrix{
\mathcal{E}:= \quad  0 \ar[r] & M \ar[r]^{i} \ar@{=}[d] & V \ar[r]^{\partial} \ar[d]_{\alpha} & L \ar[r]^{\pi} \ar[d]^{\beta} & \mathfrak{g} \ar[r] \ar@{=}[d] & 0 \\
\mathcal{E}' := \quad  0 \ar[r] & M \ar[r]_{i'} & V' \ar[r]_{\partial'} & L' \ar[r]_{\pi'} & \mathfrak{g} \ar[r] & 0
}
\]
of crossed modules. Let $s' : \mathfrak{g} \rightarrow L'$ and $q' : \text{Im} (\partial') \rightarrow V'$ be sections of $\pi'$ and $\partial'$, respectively. 
Note that $(\pi' \beta s)(x) = (\pi s) (x) = x$, for all $x \in \mathfrak{g}$. Therefore, $\beta s : \mathfrak{g} \rightarrow L'$ is another section of $\pi'$. Thus, we have
\begin{align*}
(\theta_{\mathcal{E}, s, q} &- \theta_{\mathcal{E}', \beta s, q'} ) (x,y,z) \\
= ~&[s(x), g(y,z)] - [s(y), g(x,z)] + [s(z), g (x,y)]  - g ([x,y], z) + g ([x,z], y) - g ([y,z],x) \nonumber \\
~-& [\beta s(x), g'(y,z)] + [\beta s(y), g'(x,z)] - [\beta s(z), g' (x,y)]  + g' ([x,y], z) - g' ([x,z], y) - g' ([y,z],x), \nonumber
\end{align*}
where $g'(x,y) = q' ([\beta s (x), \beta s (y)] - \beta s [x,y])$. Here we have used the same notation $[~,~]$ to denote the action of $L$ on $V$ and the action on $L'$ on $V'$. Hence, we have
\begin{align*}
(\theta_{\mathcal{E}, s, q} &- \theta_{\mathcal{E}', \beta s, q'} ) (x,y,z) \\
=~& [\beta s (x), (\alpha q - q' \beta) ([s(y), s(z)] - s[y,z]) ] 
- [\beta s (y), (\alpha q - q' \beta) ([s(x), s(z)] - s[x,z]) ] \\
~&+ [\beta s (z), (\alpha q - q' \beta) ([s(x), s(y)] - s[x,y]) ] 
- (\alpha q - q' \beta) \big(  [  s[x,y], s(z)] - s[[x,y],z]  \big) \\
~&+ (\alpha q - q' \beta) \big(  [  s[x,z], s(y)] - s[[x,z],y]  \big)
- (\alpha q - q' \beta) \big(  [  s[y,z], s(x)] - s[[y,z],x]  \big).
\end{align*}
It follows from the above expression that $(\theta_{\mathcal{E}, s, q} - \theta_{\mathcal{E}', \beta s, q'} ) (x,y,z) = (\delta \phi)(x,y,z)$, for some $\phi : \wedge^2 \mathfrak{g} \rightarrow M$. Hence, $[\theta_{\mathcal{E}, s, q}] = [\theta_{\mathcal{E}', \beta s, q'}]$ in $H^3_{CE} (\mathfrak{g}, M)$. Moreover, from the first part, we have $[\theta_{\mathcal{E}', \beta s, q'}] = [\theta_{\mathcal{E}',  s', q'}]$. Hence the class $[\theta_{\mathcal{E}, s, q}]$ does not depend on the sections $s$ and $q$. We denote the corresponding class by $[\theta_\mathcal{E}]$. Therefore, the map
\begin{align*}
 \psi : \text{Cross} (\mathfrak{g}, M) \rightarrow H^3_{CE} (\mathfrak{g}, M), ~ \mathcal{E} \rightarrow [\theta_\mathcal{E}]
\end{align*}
is well-defined. The surjectivity of $\psi$ follows from the next observation.
\end{proof}

Let $\mathfrak{g}$ be a Lie algebra and $0 \rightarrow M \rightarrow M' \xrightarrow{\pi} M'' \rightarrow 0$ an exact sequence of $\mathfrak{g}$-modules. Given an abelian extension 
\begin{align}\label{abel-extn}
0 \rightarrow M'' \rightarrow \mathfrak{e} \rightarrow \mathfrak{g} \rightarrow 0
\end{align}
of $\mathfrak{g}$ by $M''$, we consider the Yoneda product
\begin{align}\label{yoneda}
 0 \rightarrow M \rightarrow M' \xrightarrow{\mu} \mathfrak{e} \rightarrow \mathfrak{g} \rightarrow 0.
\end{align}
Writing $\mathfrak{e} = M'' \oplus \mathfrak{g}$ as a vector space, we have $\mu (m') = (\pi (m'), 0)$. An $\mathfrak{e}$-action on $M'$ is induced from the $\mathfrak{g}$-action on $M'$, namely,
\begin{align*}
[ (m'' , x) , m' ] = [x, m' ], ~~~~ \text{ for } (m'' , x) \in \mathfrak{e} \text{ and } m' \in M'.
\end{align*}
It is easy to see that (\ref{yoneda}) defines a crossed module of $\mathfrak{g}$ by $M$. It has been shown in \cite{wagemann} that the corresponding third cohomology class in $H^3_{CE} (\mathfrak{g}, M)$ as constructed above is the image of the second cohomology class in $H^2_{CE} (\mathfrak{g}, M'')$ (defined by the abelian extension (\ref{abel-extn})) under the connecting homomorphism $\partial$ in the cohomology long exact sequence
\begin{align*}
\cdots \rightarrow H^2_{CE} (\mathfrak{g}, M) \rightarrow H^2_{CE} (\mathfrak{g}, M') \rightarrow H^2_{CE} (\mathfrak{g}, M'') \xrightarrow{\partial} H^3_{CE} (\mathfrak{g}, M) \rightarrow \cdots .\\
\end{align*}

\noindent {\em Surjectivity of the map $\psi$ (of Theorem \ref{cross-3}).} The surjectivity of $\psi$ can be shown as follows \cite{wagemann}. As the category $\mathfrak{g}\text{-mod}$ possesses enough injectives, we can choose an injective $\mathfrak{g}$-module $I$ and a monomorphism $i : M \rightarrow I$. Consider the short exact sequence of $\mathfrak{g}$-modules
\begin{align}\label{inj-short}
0 \rightarrow M \xrightarrow{i} I \rightarrow M'' \rightarrow 0
\end{align}
where $M''$ is the cokernel of the map $i$. Since $I$ is injective, the cohomology long exact sequence yields $H^2_{CE} (\mathfrak{g}, M'') \cong H^3_{CE} (\mathfrak{g}, M).$ This isomorphism is given by the connecting homomorphism in the long exact sequence of cohomology.  Thus, a cohomology class $[\gamma] \in H^3_{CE} (\mathfrak{g}, M)$ corresponds to a class $[\alpha] \in H^2_{CE} (\mathfrak{g}, M'')$. We now consider an abelian extension $0 \rightarrow M'' \rightarrow \mathfrak{e}  \rightarrow \mathfrak{g} \rightarrow 0$ corresponding to the cohomology class $[\alpha] \in H^2_{CE} (\mathfrak{g}, M'')$ . It follows from the above discussion that the Yoneda product of (\ref{inj-short}) and the above abelian extension  gives rise to a crossed module whose associated third cohomology class is given by $[\gamma]$.

\section{Crossed $n$-fold extensions of Lie algebras}\label{sec-4}
Using the notion of crossed module of previous section, we introduce crossed $n$-fold extensions of a Lie algebra $\mathfrak{g}$ by a module $M$. We show that there is an abelian group structure on $\text{Opext}^n(\mathfrak{g}, M)$ of equivalence classes of crossed $n$-fold extensions of $\mathfrak{g}$ by $M$, for $n \geq 2$. The construction is similar to the case of associative algebra \cite{baues-minian}. Finally, we show that such extensions represent cohomology classes in $H^{n+1}_{CE} (\mathfrak{g}, M).$

Let $\mathfrak{g}$ be a Lie algebra and $M$ be a $\mathfrak{g}$-module. Let $n \geq 2$.

\begin{defn}
 A crossed $n$-fold extension of $\mathfrak{g}$ by $M$ is an exact sequence
\begin{align*}
0 \rightarrow M \xrightarrow{f} M_{n-1} \xrightarrow{\partial_{n-1}} \cdots \cdot \xrightarrow{\partial_2} M_1  \xrightarrow{\partial_1} L \xrightarrow{\pi} \mathfrak{g} \rightarrow 0
\end{align*}
of $\mathbb{K}$-vector spaces with the following properties:
\begin{itemize}
\item $(M_1 , L , \partial_1)$ is a crossed module with coker$(\partial_1) = \mathfrak{g},$
\item for $1 < i \leq n-1$, $M_i$ is a $\mathfrak{g}$-module and $\partial_i,~ f$ are morphisms of $\mathfrak{g}$-modules.
\end{itemize}
\end{defn}

\begin{defn}
Let
\begin{align*}
\mathcal{E} := \big( 0 \rightarrow M \xrightarrow{f} M_{n-1} \xrightarrow{\partial_{n-1}} \cdots \cdot \xrightarrow{\partial_2} M_1  \xrightarrow{\partial_1} L \xrightarrow{\pi} \mathfrak{g} \rightarrow 0  \big)
\end{align*}
and
\begin{align*}
\mathcal{E}' := \big( 0 \rightarrow M' \xrightarrow{f'} M'_{n-1} \xrightarrow{\partial'_{n-1}} \cdots \cdot \xrightarrow{\partial'_2} M'_1  \xrightarrow{\partial'_1} L' \xrightarrow{\pi'} \mathfrak{g} \rightarrow 0  \big)
\end{align*}
be two crossed $n$-fold extensions of $\mathfrak{g}$ by $M$ and $M'$, respectively. A morphism between them consists of maps $\alpha : M \rightarrow M'$, $\beta : L \rightarrow L'$ and $\delta_i : M_i \rightarrow M_i'$ for $1 \leq i \leq n-1$ such that
\begin{itemize}
\item all squares commute (see figure (\ref{cross-n-mor})),
\item $\alpha,~ \delta_i$ ($2 \leq i \leq n-1$) are morphism of $\mathfrak{g}$-modules,
\item $(\delta_1, \beta) : (M_1 , L, \partial_1) \rightarrow (M'_1 , L', \partial_1')$ is a map of crossed modules inducing the identity map on $\mathfrak{g}.$
\end{itemize}
\begin{align}\label{cross-n-mor}
\xymatrix{
 \quad  0 \ar[r] & M \ar[r]^{f} \ar@{=}[d] & M_{n-1} \ar[r]^{\partial_{n-1}} \ar[d]_{\alpha} & M_{n-2} \ar[d]_{\delta_{n-2}} \ar[r] & \ar@{.}[r] & \ar[r] & M_1 \ar[r]^{\partial_1} \ar[d]_{\delta_1} &L \ar[r]^{\pi} \ar[d]^{\beta} & \mathfrak{g} \ar[r] \ar@{=}[d] & 0 \\
 \quad  0 \ar[r] & M \ar[r]_{f'} & M'_{n-1} \ar[r]_{\partial'_{n-1}} &M_{n-2} \ar[r] &  \ar@{.}[r] &  \ar[r] & M_1' \ar[r]_{\partial'} & L' \ar[r]_{\pi'} & \mathfrak{g} \ar[r] & 0
}
\end{align}
\end{defn}

Let $\text{Opext}^n (\mathfrak{g}, M)$ be the equivalence classes of crossed $n$-fold extensions of $\mathfrak{g}$ by $M$. Then it follows that $\text{Opext}^2(\mathfrak{g}, M) = \text{Cross} (\mathfrak{g}, M).$ In the next, we show that $\text{Opext}^n (\mathfrak{g}, M)$ has a natural abelian group structure. To do that, we start with the following proposition.

\begin{prop}\label{cross-induced-by-morphism}
Let $\mathcal{E} \in \text{Opext}^n (\mathfrak{g}, M)$ be an $n$-fold extension of $\mathfrak{g}$ by $M$ and $\alpha : M \rightarrow M'$ be an $\mathfrak{g}$-module map. Then there exists an $n$-fold extension $\alpha \mathcal{E} \in \text{Opext}^n (\mathfrak{g}, M')$ and a morphism of the form $(\alpha, \delta_{n-1}, \ldots, \delta_1, \beta)$ from $\mathcal{E}$ to $\alpha \mathcal{E}$. Moreover, $\alpha \mathcal{E} \in \text{Opext}^n (\mathfrak{g}, M')$ is the unique $n$-fold extension with this property.
\end{prop}

\begin{proof}
The proof is based on pushout of $\mathfrak{g}$-modules. Let
$\mathcal{E} := (0 \rightarrow M \xrightarrow{f} M_{n-1} \xrightarrow{\partial_{n-1}} \cdots \xrightarrow{\partial_2} M_1 \xrightarrow{\partial_1} L \xrightarrow{\pi} \mathfrak{g} \rightarrow 0) \in \text{Opext}^n (\mathfrak{g}, M)$ be a crossed $n$-fold extension. We consider the $n$-fold extension $\alpha \mathcal{E} \in \text{Opext}^n (\mathfrak{g}, M)$ as 
\begin{align*}
\alpha \mathcal{E} := (   0 \rightarrow M' \rightarrow \overline{M_{n-1}} \rightarrow M_{n-2} \xrightarrow{\partial_{n-2}} \cdots \xrightarrow{\partial_2} M_1 \xrightarrow{\partial_1} L \xrightarrow{\pi} \mathfrak{g} \rightarrow 0 )
\end{align*}
where $M' \rightarrow \overline{M_{n-1}} \rightarrow M_{n-2}$ is obtained from the following pushout of $\mathfrak{g}$-modules
\[
\xymatrix{
M \ar[r] \ar[d]_{\alpha} & M_{n-1} \ar[d]^{i} \ar@/^/[ddr] &\\
M' \ar[r] \ar@/_/[drr]_0 & \overline{M_{n-1}} \ar@{.>}[rd]& \\
 & & M_{n-2}.
}
\]
Moreover $(\alpha, i, \text{id}, \ldots, \text{id})$ defines a morphism from $\mathcal{E}$ to $\alpha \mathcal{E}$.

Finally, let $\mathcal{E}' \in \text{Opext}^n (\mathfrak{g}, M')$ be an $n$-fold extension and there is a morphism $(\alpha, \delta_{n-1}, \ldots, \delta_1, \beta) : \mathcal{E} \rightarrow \mathcal{E}'$. Then by properties of pushout, one obtains a map $(1, j, \delta_{n-2}, \ldots, \delta_1, \beta) : \alpha \mathcal{E} \rightarrow \mathcal{E}'$. This shows that the class of $\alpha \mathcal{E}$ and $\mathcal{E}'$ is same in $\text{Opext}^n (\mathfrak{g}, M)$.
\end{proof}

Thus, it follows from the above proposition that a $\mathfrak{g}$-module map $\alpha : M \rightarrow M'$ induces a well-defined map
\begin{align*}
\alpha_* : \text{Opext}^n (\mathfrak{g}, M) \rightarrow \text{Opext}^n (\mathfrak{g}, M'),~ \mathcal{E} \mapsto \alpha \mathcal{E}.
\end{align*}

\begin{defn}
let $\mathcal{E} = \big( 0 \rightarrow M \xrightarrow{f} M_{n-1} \xrightarrow{\partial_{n-1}} \cdots \cdot \xrightarrow{\partial_2} M_1  \xrightarrow{\partial_1} L \xrightarrow{\pi} \mathfrak{g} \rightarrow 0  \big) \in \text{Opext}^n (\mathfrak{g}, M)$ and $\mathcal{E}' = \big( 0 \rightarrow M' \xrightarrow{f'} M'_{n-1} \xrightarrow{\partial'_{n-1}} \cdots \cdot \xrightarrow{\partial'_2} M'_1  \xrightarrow{\partial'_1} L' \xrightarrow{\pi'} \mathfrak{g} \rightarrow 0  \big) \in \text{Opext}^n (\mathfrak{g}, M)$ be two crossed $n$-fold extensions of $\mathfrak{g}$ by $M$ and $M'$, respectively. The sum of $\mathcal{E}$ and $\mathcal{E}'$ over $\mathfrak{g}$ is denoted by $\mathcal{E} \oplus_\mathfrak{g} \mathcal{E}'$ and is given by the following $n$-fold extension
\begin{align*}
0 \rightarrow M \oplus M' \rightarrow M_{n-1} \oplus M'_{n-1} \rightarrow \cdots \cdot  \rightarrow M_1 \oplus M_1' \xrightarrow{ (\partial_1, \partial'_1)} L \times_\mathfrak{g} L' \xrightarrow{q} \mathfrak{g} \rightarrow 0.
\end{align*}
Here the Lie algebra structure on $L \times_\mathfrak{g} L'$ is given by
\begin{align*}
[(x, x'), (y, y')] = ([x,y], [x',y']), ~~~ \text{ for } (x, x'), (y, y') \in L \times_\mathfrak{g} L'. 
\end{align*}
The projection $q : L \times_{\mathfrak{g}} L' \rightarrow \mathfrak{g}$ is the obvious one. The action of the Lie algebra $L \times_{\mathfrak{g}} L'$ on $M_1 \oplus M_1'$ is defined coordinatewise. It is easy to check that $(M_1 \oplus M_1' ,  L \times_\mathfrak{g} L',  (\partial_1, \partial'_1))$ defines a crossed module.
\end{defn}

\begin{defn} (Baer sum) Let $\mathcal{E}, \mathcal{E}' \in \text{Opext}^n (\mathfrak{g}, M)$ with $n \geq 3$. Then the Baer sum $\mathcal{E} + \mathcal{E}' \in \text{Opext}^n (\mathfrak{g}, M)$ is defined by
\begin{align*}
\mathcal{E} + \mathcal{E}' = \nabla_M (\mathcal{E} \oplus_\mathfrak{g} \mathcal{E}')
\end{align*}
where $\nabla_M : M \oplus M \rightarrow M, ~ (m_1 , m_2) \mapsto m_1 + m_2$ is the codiagonal map.
\end{defn}

\begin{defn}\label{zero-element} (Zero extension) Let $n \geq 3$. Then
\begin{align*}
0 \rightarrow M = M \rightarrow \underbrace{0 \rightarrow \cdots \cdot \rightarrow 0}_{n-2} \rightarrow \mathfrak{g} \rightarrow \mathfrak{g} \rightarrow 0
\end{align*}
is a crossed $n$-fold extension of $\mathfrak{g}$ by $M$. We define $0 \in \text{Opext}^n (\mathfrak{g}, M)$ to be the class of this extension.
\end{defn}

\begin{remark}\label{cross-rep-zero}
Let $\mathcal{E} := ( 0 \rightarrow M \xrightarrow{f} M_{n-1} \xrightarrow{\partial_{n-1}} \cdots \xrightarrow{\partial_2} M_1 \xrightarrow{\partial_1} L \xrightarrow{\pi} \mathfrak{g} \rightarrow 0 ) \in \text{Opext}^n (\mathfrak{g}, M)$ be a crossed $n$-fold extension, for $n \geq 3$. Suppose there is a map $g : M_{n-1} \rightarrow M$ of $\mathfrak{g}$-modules satisfying $gf = \text{id}_M$. Then there is a morphism
\begin{align*}
\xymatrix{
 \quad  0 \ar[r] & M \ar[r]^{f} \ar@{=}[d] & M_{n-1} \ar[r]^{\partial_{n-1}} \ar[d]_{g} & M_{n-2} \ar[d]_{0} \ar[r]^{\partial_{n-2}} & \ar@{.}[r] & \ar[r]^{\partial_2} & M_1 \ar[r]^{\partial_1} \ar[d]_{0} &L \ar[r]^{\pi} \ar[d]^{\pi} & \mathfrak{g} \ar[r] \ar@{=}[d] & 0 \\
 \quad  0 \ar[r] & M \ar@{=}[r] & M \ar[r]_{0} & 0 \ar[r] &  \ar@{.}[r] &  \ar[r] & 0 \ar[r]_{0} & \mathfrak{g} \ar@{=}[r] & \mathfrak{g} \ar[r] & 0
}
\end{align*}
of crossed $n$-fold extensions. This shows that the class of $\mathcal{E}$ defines $0 \in \text{Opext}^n (\mathfrak{g}, M)$.
\end{remark}

\begin{remark}
Let $\mathcal{E} := ( 0 \rightarrow M \xrightarrow{f} M_{n-1} \xrightarrow{\partial_{n-1}} \cdots \xrightarrow{\partial_2} M_1 \xrightarrow{\partial_1} L \xrightarrow{\pi} \mathfrak{g} \rightarrow 0 ) \in \text{Opext}^n (\mathfrak{g}, M)$ be a crossed $n$-fold extension, for $n \geq 3$. It follows from Proposition \ref{cross-induced-by-morphism} that $f \mathcal{E} \in \text{Opext}^n (\mathfrak{g}, M_{n-1})$ and is given by the bottom row of the following diagram
\begin{align*}
\xymatrix{
 \quad  0 \ar[r] & M \ar[r]^{f} \ar[d]_{f} & M_{n-1} \ar[r]^{\partial_{n-1}} \ar[d]_{(id, \partial_{n-1})} & M_{n-2} \ar@{=}[d] \ar[r]^{\partial_{n-2}} & \ar@{.}[r] &  \ar[r]^{\partial_1}  &L \ar[r]^{\pi} \ar@{=}[d] & \mathfrak{g} \ar[r] \ar@{=}[d] & 0 \\
 \quad  0 \ar[r] & M_{n-1} \ar[r]_{(id, 0)} & M_{n-1} \oplus M_{n-2} \ar[r]_{pr_2} & M_{n-2} \ar[r] &  \ar@{.}[r] & \ar[r]_{\partial_1} & L \ar[r]_{\pi} & \mathfrak{g} \ar[r] & 0.
}
\end{align*}
Observe that, in $f \mathcal{E}$, there is a map $g = pr_1 : M_{n-1} \oplus M_{n-2} \rightarrow M_{n-1}$ which satisfies $g \circ (id, 0) = id_{M_{n-1}}$. Hence, it follows from Remark \ref{cross-rep-zero} that $f \mathcal{E} = 0$ in $\text{Opext}^n (\mathfrak{g}, M_{n-1}).$
\end{remark}

\begin{defn}\label{inverse-element} (Inverse of an extension) Let
\begin{align*}
\mathcal{E} := \big( 0 \rightarrow M \xrightarrow{f} M_{n-1} \xrightarrow{\partial_{n-1}} \cdots \cdot \xrightarrow{\partial_2} M_1  \xrightarrow{\partial_1} L \xrightarrow{\pi} \mathfrak{g} \rightarrow 0  \big) \in \text{Opext}^n (\mathfrak{g}, M)
\end{align*}
be an extension. Then
\begin{align*}
- \mathcal{E} := \big( 0 \rightarrow M \xrightarrow{- f} M_{n-1} \xrightarrow{\partial_{n-1}} \cdots \cdot \xrightarrow{\partial_2} M_1  \xrightarrow{\partial_1} L \xrightarrow{\pi} \mathfrak{g} \rightarrow 0  \big) \in \text{Opext}^n (\mathfrak{g}, M)
\end{align*}
defines a new crossed $n$-extension of $\mathfrak{g}$ by $M$.
\end{defn}

The proof of the following theorem is straightforward.
\begin{thm}\label{thm-abelian-grp}
Let $n \geq 3$. Then the set $\text{Opext}^n (\mathfrak{g}, M) $ eqipped with the Baer sum is an abelian group. The zero element and inverse elements of this group are given by Definitions \ref{zero-element} and \ref{zero-element}, respectively.

Moreover, if $\alpha : M \rightarrow M'$ is a morphism of $\mathfrak{g}$-modules, the map $\alpha_* : \text{Opext}^n (\mathfrak{g}, M) \rightarrow \text{Opext}^n (\mathfrak{g}, M')$ is a morphism of groups.
\end{thm}

\begin{remark}
The set $\text{Opext}^2 (\mathfrak{g}, M)$ also inherits the structure of an abelian group. However, one can easily feel that the structure of this abelian group must be different than the one defined in Theorem \ref{thm-abelian-grp}.

Note that an element in $\text{Opext}^2 (\mathfrak{g}, M)$ is a equivalence class of crossed modules. We denote $0 \in \text{Opext}^2 (\mathfrak{g}, M)$ to be the class of the crossed module
$ 0 \rightarrow M = M \xrightarrow{0}  \mathfrak{g} = \mathfrak{g} \rightarrow 0$. The addition in $\text{Opext}^2 (\mathfrak{g}, M)$ is given as follows. Let $\mathcal{E} := (0 \rightarrow M \xrightarrow{i} V \xrightarrow{\partial} L \xrightarrow{\pi} \mathfrak{g} \rightarrow 0)$ and 
 $\mathcal{E}' := (0 \rightarrow M \xrightarrow{i'} V' \xrightarrow{\partial'} L' \xrightarrow{\pi'} \mathfrak{g} \rightarrow 0)$ be two crossed modules in $\text{Opext}^2 (\mathfrak{g}, M)$. Then the Baer sum $\mathcal{E} + \mathcal{E}'$ is the class of the extension
 \begin{align*}
 0 \rightarrow M \xrightarrow{j} V + V' \xrightarrow{\tilde{\partial}} L \times_\mathfrak{g} L' \xrightarrow{q} \mathfrak{g} \rightarrow 0.
 \end{align*}
 Here $V+ V'$ is the pushout of $\mathbb{K}$-vector spaces
 \[
\xymatrix{
M \oplus M \ar[r]^{i \oplus i'} \ar[d]_{\nabla_M} & V \oplus V' \ar[d]^{r} \ar@/^/[ddr]^{(\partial, \partial')} &\\
M \ar[r]_{j} \ar@/_/[drr]_0 & V + V' \ar@{.>}[rd]^{\tilde{\partial}} & \\
 & & L \times_\mathfrak{g} L'.
}
\]
The $(L \times_\mathfrak{g} L')$-module structure on $V + V'$ is induced from the module structure on $V \oplus V'$ via the quotient map $r : V \oplus V' \rightarrow V+ V'$. Namely, the action is given by $[(x, x'), r (v,v')] = r ([x,v], [x',v'])$, for $(x, x') \in L \times_\mathfrak{g} L'$ and $(v, v') \in V \oplus V'$. With this module structure on $V + V'$, it is easy to show that $(V + V' , L \times_{\mathfrak{g}} L', \tilde{\partial})$ is a crossed module.
\end{remark}

Next we associate a long exact sequence to any short exact sequence of $\mathfrak{g}$-modules. This will help us to prove our main theorem. Let
\begin{align*}
0 \rightarrow M \xrightarrow{\alpha} M'  \xrightarrow{\beta} M'' \rightarrow 0
\end{align*}
be a short exact sequence of $\mathfrak{g}$-modules. For any $n \geq 2$, we define a homomorphism
\begin{align*}
\delta : \text{Opext}^n (\mathfrak{g}, M'') \rightarrow \text{Opext}^{n+1} (\mathfrak{g}, M)
\end{align*}
as follows. Given an extension 
\begin{align*}
\mathcal{E} = \big(  0 \rightarrow M'' \xrightarrow{f} M_{n-1} \xrightarrow{\partial_{n-1}} \cdots \cdot \xrightarrow{\partial_2} M_1 \xrightarrow{\partial_1} L \xrightarrow{\pi} \mathfrak{g} \rightarrow 0 \big) \in \text{Opext}^n (\mathfrak{g}, M''),
\end{align*}
we define $\delta \mathcal{E} \in \text{Opext}^{n+1} (\mathfrak{g}, M)$ to be the class of the extension
\begin{align*}
0 \rightarrow M \xrightarrow{\alpha} M' \xrightarrow{f \beta} & M_{n-1} \xrightarrow{\partial_{n-1}} \cdots \cdot \xrightarrow{\partial_2} M_1 \xrightarrow{\partial_1} L \xrightarrow{\pi} \mathfrak{g} \rightarrow 0.
\end{align*}
Then $\delta$ is a well-defined homomorphism  for $n \geq 2$.

The proof of the following theorem is similar to the proof of Baues and Minian for associative algebra case \cite{baues-minian}, hence we omit the details.
\begin{thm}\label{short-long-seq}
A short exact sequence
\begin{align*}
0 \rightarrow M \xrightarrow{\alpha} M'  \xrightarrow{\beta} M'' \rightarrow 0
\end{align*}
of $\mathfrak{g}$-modules induces a long exact sequence of abelian groups ($n \geq 2$)
\begin{align*}
\text{Opext}^n (\mathfrak{g}, M) \xrightarrow{\alpha_*} \text{Opext}^n (\mathfrak{g}, M') \xrightarrow{\beta_*} \text{Opext}^n (\mathfrak{g}, M'') \xrightarrow{\delta} \text{Opext}^{n+1} (\mathfrak{g}, M) \rightarrow \cdots .
\end{align*}
\end{thm}

Now we are in a position to prove our main theorem.
\begin{thm}
For any $n \geq 2$, there exists an isomorphism of abelian groups
\begin{align*}
\text{Opext}^n (\mathfrak{g}, M) \cong H^{n+1}_{CE} (\mathfrak{g}, M).
\end{align*}
\end{thm}

\begin{proof}
The result is true for $n = 2$ (see Theorem \ref{cross-3}). Let $n \geq 3$. We consider a short exact sequence
\begin{align*}
0 \rightarrow M \xrightarrow{\alpha} M'  \xrightarrow{\beta} M'' \rightarrow 0
\end{align*}
of $\mathfrak{g}$-modules with $M'$ injective. Then by Theorem \ref{short-long-seq}, it follows that
\begin{align*}
\text{Opext}^n (\mathfrak{g}, M'') \cong \text{Opext}^{n+1} (\mathfrak{g}, M).
\end{align*}
Moreover, it follows from the cohomology long exact exact sequence that $H^{n+1}_{CE} (\mathfrak{g}, M'') \cong H^{n+2}_{CE} (\mathfrak{g}, M)$. Therefore,
\begin{align*}
\text{Opext}^3 (\mathfrak{g}, M) \cong~ \text{Opext}^2 (\mathfrak{g}, M'') 
\cong~& H^3_{CE} (\mathfrak{g}, M'') \qquad(\text{by Theorem }\ref{cross-3})\\
\cong~& H^4_{CE} (\mathfrak{g}, M).
\end{align*}
We conclude the result by using the induction on $n$.
\end{proof}

\begin{remark}
Lie-Rinehart algebras are algebraic analogue of Lie algebroids and closely related to Lie algebras. These algebras pay special attention due to its connection Poisson geometry \cite{hueb2}. In \cite{casas-ladra-pira} the authors studied crossed modules for Lie-Rinehart algebras and classify them by the third cohomology of Lie-Rinehart algebras. However, their crossed modules are similar to the traditional one for Lie algebras. It would be interestiong to study crossed extensions of Lie-Rinehart algebras and their classification in terms of higher cohomologies of Lie-Rinehart algebras.
\end{remark}

\section{The case of a Leibniz algebra}
Leibniz algebra was introduced by Loday in connection with cyclic homology and Hochschild homology of matrix algebras \cite{loday}. The cohomology theory of Liebniz algebras was introduced by the same author. In this section, we outline that Leibniz cohomology groups are also represented by crossed extensions of Leibniz algebras.

A (right) Leibniz algebra is a $\mathbb{K}$-vector space $\mathfrak{h}$ together with a $\mathbb{K}$-bilinear map $[~,~] : \mathfrak{h} \times \mathfrak{h} \rightarrow \mathfrak{h}$ satisfying
\begin{align*}
[x,[y,z]] = [[x,y],z] - [[x,z],y], ~~~ \text{ for all } x, y, z \in \mathfrak{h}.
\end{align*}

Let $(\mathfrak{h}, [~,~])$ be a Leibniz algebra. An $\mathfrak{h}$-module is a vector space $M$ together with two bilinear maps (both of them denoted by $[~,~]$) $\mathfrak{h} \times M \rightarrow M$ and $M \times \mathfrak{h} \rightarrow M$ satisfying
\begin{align*}
[x,[y,z]] = [[x,y],z] - [[x,z],y],
\end{align*}
whenever one of the variable is from $M$ and the others two are from $\mathfrak{h}$.

The cohomology of the Leibniz algebra $\mathfrak{h}$ with coefficients in $M$  is the cohomology of the complex $(C^n_{Leib} (\mathfrak{h}, M), \delta)_{n \geq 0}$ where $C^0_{Leib} (\mathfrak{h}, M) = M$ and $C^n_{Leib} (\mathfrak{h}, M) = Hom_{\mathbb{K}} (\mathfrak{h}^{\otimes n} , M)$, for $n \geq 1$. The differential $\delta$ is given by $(\delta m)(x) = [x,m]$, for $m \in M$, $x \in \mathfrak{h}$ and
\begin{align*}
(\delta f) (x_1, \ldots, x_{n+1}) =~& [x_1, f(x_2, \ldots, x_{n+1}) ] + \sum_{i=2}^{n+1} (-1)^i [ f(x_1, \ldots, \widehat{x_i}, \ldots, x_{n+1}), x_i ] \\
~& + \sum_{1 \leq i < j \leq n+1} (-1)^{j+1} f(x_1, \ldots, x_{i-1}, [x_i, x_j], x_{i+1}, \ldots, \widehat{x_j}, \ldots, x_n).
\end{align*}
Similar to the case of Chevalley-Eilenberg cohomology, the lower degree Leibniz cohomology has the following interpretations. The zero-th cohomology group $H^0_{Leib} (\mathfrak{h}, M)$ is the submodule of invariants on $M$ and the first cohomology group
\begin{align*}
H^1_{Leib} (\mathfrak{h}, M) = \frac{\{ f : \mathfrak{h} \rightarrow M|~ f[x,y] = [x,f(y)] + [f(x), y], \forall x , y \in \mathfrak{h}\}}{\{\delta m |~ m \in M \}}
\end{align*}
is the space of outer derivations. The second cohomology group $H^2_{Leib} (\mathfrak{h}, M)$ classifies the equivalence classes of extensions of the Leibniz algebra $\mathfrak{h}$ by $M$.

Crossed modules over a Leibniz algebra can be defined in a similar way. One only needs to careful about non skew-symmetric version of the identities used in the case of Lie algebra.

A crossed module over a Leibniz algebra is a triple $(V, L, \partial)$ in which $L$ is a Leibniz algebra, $V$ is a $L$-module and $\partial : V \rightarrow L$ is a map of $L$-modules satisfying
\begin{align*}
[\partial v, w ] = [ v, \partial w], ~~~~ \text{ for all } v, w \in L.
\end{align*}
In a similar way, a crossed module yields a Leibniz algebra structure on $\mathfrak{h}= \text{coker} (\partial)$ and there is an exact sequence
\begin{align*}
0 \rightarrow M \rightarrow V \xrightarrow{\partial} L \xrightarrow{\pi} \mathfrak{h} \rightarrow 0
\end{align*}
where $M = \text{ker}(\partial)$. Moreover, there is a Leibniz algebra action of $\mathfrak{h}$ on $M$. A morphism of crossed modules over Leibniz algebras can be defined in a similar way.

In this case, one can also prove that equivalence classes of crossed modules with cokernel $\mathfrak{h}$ and kernel $M$ are in one-to-one correspondence with $H^3_{\text{Leib}} (\mathfrak{h}, M)$. The proof is similar to the Lie algebra case (Theorem \ref{cross-3}). However, in this case, the Leibniz $3$-cocycle $\theta_{\mathcal{E}, s, q}: \mathfrak{h}^{\otimes 3} \rightarrow M$ is given by
\begin{align*}
\theta_{\mathcal{E}, s, q} (x,y,z) = [s(x), g (y,z)] + [ g(x,z), s(y)] - [g(x,y), s(z)] - g([x,y], z) + g([x,z], y) + g (x, [y,z]).
\end{align*}
Associated to any third cohomology class in $H^3_{\text{Leib}} (\mathfrak{h}, M)$ the existence of the corresponding crossed module can be shown by the way that have been described in Theorem \ref{cross-3}.

Moreover, one can define crossed $n$-fold extension of a Leibniz algebra $\mathfrak{h}$ by a module $M$. Along the lines of Section \ref{sec-4}, we obtain the following theorem for Leibniz algebras.

\begin{thm}
The set $\text{Opext}^n (\mathfrak{h}, M)$ of equivalence classes of crossed $n$-fold extensions of $\mathfrak{h}$ by $M$ forms an abelian group, for $n \geq 2$. Moreover, there exists an isomorphism of abelian groups $\text{Opext}^n (\mathfrak{h}, M) \cong H^{n+1}_{\text{Leib}} (\mathfrak{h}, M).$
\end{thm}

%The free Lie algebra $L(V)$ over the vector space $V$ is a subspace of $T(V)$ generated by $V$ under the bracket operation. Let $D : V^{\otimes n} \rightarrow V^{\otimes n}$ (Dynkin bracketing) be the map defined by
%\begin{align*}
%D (v_1 \otimes \cdots \otimes v_n) := [[ \cdots [[ v_1, v_2],v_3], \cdots], v_n].
%\end{align*}
%An element $\omega \in V^{\otimes n}$ lies in $L(V)$ if and only if $D (\omega) = n \omega$. Note that $L(V)$ is a graded vector space induced from gradition of $T(V)$.

%The Chevalley-Eilenberg cochain complex can also be described as $C^{n}_{CE} (\mathfrak{g}, M) = \text{Hom}_{\mathbb{K}} (\wedge^n \mathfrak{g}, M) = \text{Hom}_{U(\mathfrak{g})} (U(\mathfrak{g}) \otimes_{\mathbb{K}} \wedge^n \mathfrak{g}, M)$ and the differential is induced from the Koszul resolution
%\begin{align*}
%\rightarrow U(\mathfrak{g}) \otimes_{\mathbb{K}} \wedge^3 \mathfrak{g} \xrightarrow{d_3} U(\mathfrak{g}) \otimes_{\mathbb{K}} \wedge^2 \mathfrak{g} \xrightarrow{d_2} U(\mathfrak{g}) \otimes_{\mathbb{K}} \mathfrak{g} \xrightarrow{d_1} U(\mathfrak{g}) \xrightarrow{\epsilon} \mathbb{K} \rightarrow 0.
%\end{align*}

\medskip

\noindent {\bf Acknowledgement.} The research is supported by the fellowship
of Indian Institute of Technology, Kanpur (India). The author would like to thank the Institute for their support.

\end{document}